\documentclass{article}
\usepackage[T1]{fontenc}
\usepackage{courier}
\usepackage{amsmath}
\usepackage{amsfonts}
\usepackage{amsthm,mathtools}
\usepackage{accents}
\usepackage{tikz-cd}
\usepackage[hyperfootnotes=false, colorlinks, linkcolor={blue}, citecolor={magenta}, filecolor={blue}, urlcolor={blue}]{hyperref}
\newcommand{\ds}[1]{\displaystyle{#1}}
\newcommand{\mr}[1]{\mathrm{#1}}
\newcommand{\mb}[1]{\mathbb{#1}}
\newcommand{\mc}[1]{\mathcal{#1}}

\newtheorem{lemma}{Lemma}
\newtheorem{thm}{Theorem}
\newtheorem*{theorem}{Theorem}

\newtheorem*{remark}{Remark}
\newtheorem{prop}{Proposition}

\newcommand{\HH}{\mathfrak h}

\renewcommand{\H}{\mathbb H}

\newcommand{\Q}{{\mathbb Q}}
\newcommand{\Z}{{\mathbb Z}}
\newcommand{\N}{{\mathbb N}}
\newcommand{\R}{{\mathbb R}}

\newcommand{\C}{{\mathbb C}}

\newcommand{\GL}{{\rm GL}}

\newcommand{\GSp}{{\rm GSp}}
\newcommand{\Sp}{{\rm Sp}}

\newcommand{\cH}{{\cal H}}

\newcommand{\SSp}{{\rm Sp}}

\newcommand{\mat}[4]{{\setlength{\arraycolsep}{0.5mm}\left[
\begin{smallmatrix}#1&#2\\#3&#4\end{smallmatrix}\right]}}
\newcommand{\forget}[1]{}

\def\qdots{\mathinner{\mkern1mu\raise0pt\vbox{\kern7pt\hbox{.}}\mkern2mu
\raise3.4pt\hbox{.}\mkern2mu\raise7pt\hbox{.}\mkern1mu}}

\begin{document}
\begin{title}
{\Large\bf Hecke eigenvalues of Ikeda lifts}
\end{title}
\author{Nagarjuna Chary Addanki and Ameya Pitale}
\maketitle
\begin{abstract}
In this paper, we study the Hecke eigenvalues of Ikeda lifts. Using the spherical map for the Hecke algebra of the symplectic group, we obtain an explicit formula for the eigenvalues $\lambda_F(p^r)$. From this formula, we show that $\lambda_F(p^r)$ can be written as a polynomial in $p^{\pm 1/2}$ with a positive leading term. Furthermore, we prove that the coefficients of this polynomial are bounded and, as a consequence, the Hecke eigenvalues $\lambda_F(p^r)$ are positive for all sufficiently large primes $p$. 
\end{abstract}

\tableofcontents

\section{Introduction}  Let $k$ and $n$ be positive integers and let $\Gamma_n$ denote the symplectic group $\Sp_{2n}(\Z)$. The space of Siegel cusp forms, $S_k(\Gamma_n)$ is a finite-dimensional vector space that admits a basis consisting of simultaneous eigenfunctions of the symplectic Hecke algebra $\mc{H}_n$. These eigenfunctions are called Hecke eigenforms. If $F$ is an eigenform, for each Hecke operator $T$, let $\lambda_F(T)$ denote the corresponding Hecke eigenvalue.

For a prime $p$, one can use the Satake $p$-parameters of $F$ to find a formula for $\lambda_F(T)$. This is done as follows: there is a map $\Omega$, called the spherical map, from $\cH_{n,p}$ (the local Hecke algebra at the prime $p$) to $\Q[x_0^\pm, \cdots, x_n^\pm]_{W_n}$, which is the space of polynomials in $x_0, \cdots, x_n$ with coefficients in $\Q$, which are invariant under the action of the symplectic Weyl group $W_n$. Given a Siegel Hecke eigenform $F$, the Satake $p$-parameters are $n+1$ complex numbers $\alpha_{0,p}, \cdots, \alpha_{n,p}$ (depending on $F$ and $p$) such that the eigenvalue $\lambda_F(T)$ of $T \in \cH_{n,p}$ is given by $\Omega(T)(\alpha_{0,p}, \cdots, \alpha_{n,p})$. 

The action of Hecke operators on $S_k^n$ is self-adjoint with respect to the Petersson inner product, and hence the Hecke eigenvalues are real numbers. The signs of these Hecke eigenvalues seem to encode valuable information regarding the underlying Siegel modular forms. For example, for genus $n=1$, there is a result by Kowalski et al (see \cite{KLSW10}) that two elliptic modular forms having the same signs for the Hecke eigenvalues must be multiples of each other. There are partial generalizations of this result to genus $n=2$ (see \cite{Add-1}, \cite{GKP21} and \cite{KMS22}). 

Another application of the signs of Hecke eigenvalues is the result of Breulmann \cite{Br99} which states that a genus $2$ Siegel cusp form is a Saito-Kurokawa lift if and only if all its Hecke eigenvalues are positive. Saito-Kurokawa lifts are lifts from elliptic modular forms to Siegel modular forms of genus $2$ and their standard $L$-function can be written as a product of shifted $L$-functions of the elliptic modular form.  Ikeda \cite{Ik01} obtained a generalization of the Saito-Kurokawa lift to Siegel cusp forms of genus $2n$. A natural question is whether the Hecke eigenvalues of Ikeda lifts are also all positive. The eigenvalue $\lambda_F(p)$ corresponding to the Hecke operator $T(p)$ is given in terms of symmetric functions of the Satake $p$-parameters of the Siegel modular forms. Writing the Satake $p$-parameters of the Ikeda lift in terms of those of the elliptic modular forms, \cite{Add25} and  \cite{GN25} show that $\lambda_F(p) > 0$ for Ikeda lifts. For $r \geq 1$, let $T(p^r) = \sum_{\mu(g) = p^r} \Gamma_n g \Gamma_n$ and let $\lambda_F(p^r)$ be the Hecke eigenvalue for the operator $T(p^r)$. For Ikeda lifts, to show the positivity of $\lambda_F(p^r), r > 1$, one tool is the Dirichlet series relation (see \cite{An74})
$$\sum\limits_{r=0}^\infty \frac{\lambda_F(p^r)}{p^{rs}} = \frac{P(p^{-s})}{Q(p^{-s})},$$
where $Q(p^{-s})^{-1} = L_p(F, s, {\rm spin}),$ the local spin $L$-factor for $F$ and $P(x)$ is a polynomial of degree $2^n-2$. This was used by Breulmann in \cite{Br99} for the genus $2$ case. In \cite{Sc03}, Schmidt has obtained the explicit formula for the local factor of the spin $L$-function for the Ikeda lift of all genera. Unfortunately, the polynomial $P(x)$ in the numerator is known explicitly only for genus at most $4$ (see \cite{Van11}). Using this explicit expression, the first author proved positivity of $\lambda(p^r)$, for $p$ large enough, for Ikeda lifts of genus $4$ (see \cite{Add25}). 

In this paper, we adopt a different approach which avoids the use of the Dirichlet series numerator. Instead, we compute the eigenvalues $\lambda_F(p^r)$ by explicitly determining the spherical image of the Hecke operators $T(p^r)$. The spherical image for the symplectic group contains a component corresponding to the spherical image of the general linear group. An explicit formula for the spherical image for the general linear group case is computed in \cite{An70}. Using this result, we compute $\Omega(T(p^r))$ in Proposition \ref{O(T(p))-prop}. In general, the formula for $\Omega(T(p^r))$ is very complicated involving a sum over the symmetric group $S_n$. The key idea is that when we substitute the Satake $p$-parameters of Ikeda lifts into the formula for $\Omega(T(p^r))$, all terms excepting the identity element of $S_n$ vanish, thus greatly simplifying the formula. 

\begin{theorem} Let $F \in S_{k+n}(\Gamma_{2n})$ be the Ikeda lift of a Hecke eigenform $f \in S_{2k}(\Gamma_1)$. If $a_p,a_p^{-1}$ are the Satake $p$-parameters of $f$ then
     \[\lambda_F(p^r) = p^{r(nk-\frac{n}{2})}\sum_{0 \leq \delta_1 \leq \dots \delta_{2n} \leq r } a_p^{-nr+\sum \delta_i} p^{\sum (-n - \frac{1}{2}+i)\delta_i}  \frac{\varphi_{2n}(p^{-1})}{\varphi_{k_1}(p^{-1})\dots \varphi_{k_t}(p^{-1})} \]
\end{theorem} 
See Eq (\ref{delta-k}) and (\ref{varphi}) for the definitions of $\varphi$ and the $k_i's$. From this formula, we deduce that $\lambda_F(p^r)$ is a polynomial in $p^{\pm 1/2}$ with coefficients in $\Z[a_p,a_p^{-1}]$. The fact that the leading coefficient is $1$ and the other coefficients are bounded implies that $\lambda_F(p^r)$ is positive for all $p$ sufficiently large. In fact, using the precise information  regarding the sizes of these coefficients obtained in Proposition \ref{lamda_F(p^r)}, we are able to get a lower bound for the primes $p$ for which the positivity result holds.
\begin{theorem}
     Let $F \in S_{k+n}(\Gamma_{2n})$ be the Ikeda lift of a Hecke eigenform $f \in S_{2k}(\Gamma_1)$. Let $\lambda_F(p^r)$ be the eigenvalue of $F$ for the Hecke operator $T(p^r)$.
   For $r \geq 1$, we have $\lambda_F(p^r)>0$ for all $p >14n^2$.
\end{theorem} 
The positivity of Hecke eigenvalues of Ikeda lifts is currently known only in a few cases. The positivity of $\lambda_F(p)$ is proved in \cite{Add25} and \cite{GN25} for Ikeda lifts of all genera. In \cite{Add25}, it is shown that, for Ikeda lift of genus 4, $\lambda_F(p^r)$ is positive for all sufficiently large primes $p$. However, the corresponding lower bound on $p$ is not made explicit. To the best of our knowledge, the above Theorem provides the most general positivity result currently available for Hecke eigenvalues of Ikeda lifts, as it applies in arbitrary genus and gives an explicit lower bound on $p$. Let us remark here that we have not tried to obtain an optimal lower bound for $p$ in the above theorem, and it is possible that it is further improved by more detailed analysis of constants involved in the proof of the Theorem. It is, of course, expected that the positivity is valid for all primes $p$, but we cannot show that at this point.\\
\textbf{Outline of the Paper:} In Section \ref{hecke-alg} we review the Hecke algebras of the general linear and symplectic groups and recall their action on Siegel modular forms. In Section \ref{sph-map}, we compute the spherical image of $T(p^r)$. In the final section we apply this formula to Ikeda lifts to obtain explicit expressions for the eigenvalues $\lambda_F(p^r)$ and deduce the positivity result. \\
\textbf{Acknowledgments:} The authors thank the referee for their valuable comments and suggestions, which improved the exposition and clarity of the paper.

\section{Hecke Algebra}\label{hecke-alg}
In this section we describe the Hecke algebra for the symplectic group and its action on Siegel modular forms. We direct the reader to \cite{An09} and \cite{AZ95} for the details on Hecke algebra presented below. 
\subsection*{Abstract Hecke algebra} 
Let $G$ be a group, $\Gamma$ be a subgroup of $G$ and $S$ be a subset of $G$ that is closed under multiplication. $(\Gamma,S)$ is called a Hecke pair if $[\Gamma : \Gamma \cap g^{-1} \Gamma g ] < \infty$ and $[ g^{-1} \Gamma g : \Gamma \cap g^{-1} \Gamma g ] < \infty$ for every $g \in S.$ For a Hecke pair we define $L(\Gamma,S)$ to be the free $\Z$-module generated by elements of the form $\Gamma g$ for $g \in S.$   The Hecke algebra associated with $(\Gamma, S)$ is defined as the set  
\[
\{ \sum a_i \Gamma g_i \in L(\Gamma,S) \otimes_{\Z} \Q : \sum a_i \Gamma g_i \gamma = \sum a_i \Gamma g_i \ \text{for all} \ \gamma \in \Gamma \}. 
\] If 
$$ t = \sum_{i} a_i (\Gamma g_i) \ \text{and} \ t' = \sum_{j} b_j (\Gamma h_j)$$ then 
$$ t . t' \coloneqq \sum_{i,j} a_ib_j (\Gamma g_i h_j).$$ It follows from the conditions on $S$ that for any $g \in S$, there exist finitely many $g_i$ such that $\Gamma g \Gamma = \sqcup \Gamma g_i.$ For every $g \in S$, we define the element $\Gamma g \Gamma \in L(\Gamma , S) $ as $\sum \Gamma g_i.$ The Hecke algebra associated with the pair $(\Gamma,S)$ is generated by $ \{ \Gamma g \Gamma : g \in S \}.$ In this paper we work with Hecke algebras of the general linear group and the symplectic group. 

Let $G^n = \GL_n(\Q)$ and $\Lambda_n= \GL_n(\Z).$ $(\Lambda_n, G^n)$ is a Hecke pair and the associated Hecke algebra is denoted by $\H_n.$ 
Let $\mb{H}_{n,p}$ represent the sub-algebra of $\mb{H}_n$ generated by those $\Lambda_n g \Lambda_n $ with $ g \in \mr{GL}_n(\mb{Z}[p^{-1}]).$ Further, let $\underline{\mb{H}}_{n,p}$ represent the sub-algebra generated by $\Lambda_n g \Lambda_n $ with $g \in M_n(\mb{Z}) \cap  \mr{GL}_n(\mb{Z}[p^{-1}]).$ Then we have 
\[
\H_n = \otimes' \H_{n,p}.
\] 
For a ring $R$, let the symplectic group of similitudes of genus $n$ be defined by
$$
\GSp_{2n}(R) := \{g \in \GL_{2n}(R) : {}^{t}g J_n g = \mu(g) J_n,
\mu(g) \in \GL_1(R) \} \qquad  \mbox{ where } J_n =
\mat{}{I_n}{-I_n}{}.
$$
 Let $\SSp_{2n}(R)$ be the
subgroup with $\mu(g)=1$. Let $\GSp_{2n}(\Q)^+$ be the subgroup of elements $g \in \GSp_{2n}(\Q)$ with $\mu(g) > 0$. If $G_n := \GSp_{2n}(\Q)^+ \cap M_{2n}(\Z)$ and $\Gamma_n = \Sp_{2n}(\Z)$, then $(\Gamma_n,G_n)$ forms a Hecke pair and the Hecke algebra associated with it is denoted by $\mc{H}_n$. For a prime number $p$, let $G_{n,p} = \{g \in G_n : \mu(g) = p^k, k \in \Z\}$. Let $\mathcal H_{n,p}$ be the Hecke algebra associated with the Hecke pair $(\Gamma_n, G_{n,p})$. Then we have $\mathcal H_n = \otimes' \mathcal H_{n,p}$. The action of $\mc{H}_n$ on the space of Siegel modular forms is explained in the next section. 

\subsection*{Action on Siegel modular forms}

The group $\GSp^+_{2n}(\R) := \{ g \in
\GSp_{2n}(\R) : \mu(g) > 0 \}$ acts on the Siegel upper half
space $\HH_n := \{ Z \in M_n(\C) : {}^{t}Z = Z, {\rm Im}(Z) > 0
\}$ by
$$
g \langle Z \rangle := (AZ+B)(CZ+D)^{-1} \qquad \mbox{ where } g =
\mat{A}{B}{C}{D} \in \GSp^+_{2n}(\R), Z \in \HH_n.
$$
For a positive integer $k,$ we define the slash operator $|_k$ acting on holomorphic functions $F$ on $\HH_n$ by
\begin{equation}\label{slash-operator-defn}
(F|_kg)(Z) := \mu(g)^{\frac{nk}2-\frac{n(n+1)}2} \det(CZ+D)^{-k} F(g \langle
Z \rangle) \qquad \mbox{ where } g = \mat{A}{B}{C}{D} \in
\GSp^+_{2n}(\R), Z \in \HH_n.
\end{equation}
Let $M_k^n$ (resp. $S_k^n)$ be the space of
holomorphic Siegel modular (resp. cusp) forms of weight $k$, genus $n$ with respect to $\Gamma_n$ (see Definitions 1.8 and 1.13 from \cite{Pit19}). Then $F \in M_k^n$
satisfies $F |_k \gamma = F$ for all $\gamma \in \Gamma_n$.

We now describe the Hecke operators acting on $M_k^n$.  For $g \in G_n$, let $T(g) := \Gamma_n g \Gamma_n = \sqcup_i \Gamma_n g_i \in \mathcal H_n$. The Hecke algebra $\mathcal H_n$ acts on $M_k^n$ (or on $S_k^n$) as follows. Let $F \in M_k^n$. Then
\begin{equation}\label{Hecke-action}
T(g)F := \sum\limits_i F |_k g_i.
\end{equation}
By Theorem 4.7 of \cite{An09}, we know that $M_k^n$ (and $S_k^n$) has a basis of simultaneous eigenfunctions of the Hecke algebra $\mathcal H_n$.  Let $F \in M_k^n$ be a Hecke eigenform, and let $T(g)F = \lambda(g) F$, where $\lambda(g)$ are the Hecke eigenvalues. For any prime number $p$, it is known that there
are $n+1$ complex numbers $\alpha_{0,p}, \alpha_{1,p}, \cdots,
\alpha_{n,p}$ depending on $F$, with the following property. If $g$ satisfies $\mu(g) = p^r$, then
\begin{equation}\label{classical-satake-parameters}
\lambda(g) =\alpha_{0,p}^r \sum_i
\prod\limits_{j=1}^n (\alpha_{j,p}p^{-j})^{d_{ij}},
\end{equation}
where $\Gamma_n g \Gamma_n = \bigsqcup_i \Gamma_n
g_i$, with
$$g_i = \mat{A_i}{B_i}{0}{D_i} \quad \mbox{ and } \quad D_i =
\begin{bmatrix}p^{d_{i1}}&&\ast\\&\ddots&\\0&&p^{d_{in}}\end{bmatrix}.$$
The
$\alpha_{0,p}, \alpha_{1,p}, \cdots, \alpha_{n,p}$ are called the {\it classical Satake
$p$-parameters} of the eigenform $F$. 
\section{Spherical map on Hecke Algebras}\label{sph-map} 
To compute the eigenvalues of $T(p^r)$, we first describe the spherical map for the symplectic Hecke algebra. The strategy is to relate it to the spherical map for the general linear group, for which explicit formulas are known.

Let $S_n$ be the permutation group of order $n$. Let $W_n$ represent the set of automorphisms on the ring $\mb{Q}[x_0^{\pm},x_1^{\pm},\dots,x_n^{\pm}]$ generated by the following elements 
\begin{enumerate}
    \item For $\sigma \in S_n$, we have  $\sigma(x_i)=  x_{\sigma(i)} $ for $i = 1, \cdots, n$ and $\sigma(x_0)=x_0$.
    \item For $i=1,2,\dots,n$. $\tau_i(x_0)=x_0x_i$, $\tau_i(x_i) = x_i^{-1}$ and $\tau_i(x_j) = x_j$ for all $ j \notin \{0, i\}$.
\end{enumerate} Theorems 2.20 and 3.30 of Chapter 3 of \cite{AZ95} define two isomorphisms, called the spherical maps, 
$$ \Omega : \mathcal{H}_{n,p} \longrightarrow \mb{Q}[x_0^{\pm},x_1^{\pm},\cdots ,x_n^{\pm}]_{W_n}, \quad  \omega : \mathbb{H}_{n,p} \longrightarrow \mb{Q}[x_1^{\pm},\cdots ,x_n^{\pm}]_{S_n}.$$
Here, $\mb{Q}[x_0^{\pm},x_1^{\pm},\cdots ,x_n^{\pm}]_{W_n}$ is the ring of all $W_n$-invariant polynomials in $x_0^{\pm}, x_1^{\pm}, \cdots, x_n^{\pm}$ over $\Q$ and $\mb{Q}[x_1^{\pm},\cdots ,x_n^{\pm}]_{S_n}$ is the ring of all $S_n$-invariant polynomials in $x_1^{\pm}, \cdots, x_n^{\pm}$ over $\Q$.

The significance of the spherical map is that it allows us to obtain the eigenvalue of the Hecke operators on Siegel Hecke eigenforms in terms of their Satake parameters. Suppose $F \in M_k^n$ is a Siegel Hecke eigenform with Satake $p$-parameters $\alpha_{0,p}, \alpha_{1,p}, \cdots, \alpha_{n,p}$. For $T \in  \mathcal{H}_{n,p}$, let $\Omega(T) = f_T(x_0, x_1, \cdots, x_n) \in \mb{Q}[x_0^{\pm},x_1^{\pm},\dots ,x_n^{\pm}]_{W_n}$. Then, we have
$$TF = \lambda(T) F, \text{ where } \lambda(T) = f_T(\alpha_{0,p}, \alpha_{1,p}, \cdots, \alpha_{n,p}).$$
\subsection*{Spherical image of $T(p^r)$}
For a positive integer $r,$ $T(p^r)$ is defined to be $\ds{\sum_{\mu(g)=p^r}\Gamma g \Gamma}.$ From \cite[Pg. 150]{AZ95} we have the formula for the spherical image of Hecke operators: \begin{equation}\label{T(p^r)2} \Omega(T(p^r)) = x_0^r\sum_{0 \leq \delta_1 \leq \dots \delta_n \leq r } p^{\sum (n-i+1)\delta_i} \omega(t(p^{\delta_1},\dots,p^{\delta_n})). 
\end{equation} Here, $t(p^{\delta_1},\dots, p^{\delta_n})$ represents the double coset $\Lambda_n \mr{diag}(p^{\delta_1},\dots, p^{\delta_n}) \Lambda_n$. An explicit formula for $\omega(t(p^{\delta_1},\dots,p^{\delta_n}))$ is computed in \cite{An70} using an alternate interpretation of $\mb{H}_{n,p}$.

The Hecke algebra also admits a realization as convolution algebra of functions. Let $\Q_p$ be the field of $p$-adic numbers and $\Z_p$ be its ring of integers. Let $G_p = \mr{GL}_n(\mb{Q}_p)$, $K_p = \mr{GL}_n(\mb{Z}_p)$ and $D(G_p,K_p)$ denote the space of continuous, compactly supported functions $f : G_p \rightarrow \mb{C}$ satisfying $f(\gamma_1 g \gamma_2) = f(g)$ for all $g \in G_p$ and $\gamma_1, \gamma_2 \in K_p.$ The product is given by the convolution:  \[ (f_1 \ast f_2) (x) = \int_{G_p} f_1(xy^{-1})f_2(y)dy. \] Here, the Haar measure is normalized such that $\mr{vol}(K_p)=1$. For $g \in G_p$, let $\chi_g$ represent the characteristic function of the set $K_p g K_p$. 
The two definitions of the Hecke algebra are connected via the isomorphism $\Psi : \mb{H}_{n,p} \rightarrow D(G_p,K_p)$ where $\Psi(\Lambda_n g \Lambda_n) = \chi_g .$ 

Each structure admits a spherical map and they can be connected by observing the images. $\underline{\H}_{n,p}$ is generated by elements of the form $\pi_i = \Lambda_n t(\underbrace{1,\dots,1}_{n-i},\underbrace{p,\dots,p}_i) \Lambda_n$, for $i=1,\cdots,n$. Restricting $\omega$ to $\underline{\mb{H}}_{n,p}$, we obtain the isomorphism \[
\omega : \underline{\mb{H}}_{n,p} \rightarrow \mb{Q}[\omega(\pi_1),\dots,\omega(\pi_n)].
\]Moreover, by \cite[Lemma 2.21, Chapter 3]{AZ95} \[ \omega(\pi_i) = p^{-\frac{i(i+1)}{2}} s_i(x_1, \dots ,x_n),\] 
where $s_i(x_1, \dots , x_n)$ denotes the $i^{\text{th}}$ elementary symmetric function in $x_1, \dots ,x_n$. 

On the other hand, $\Psi(\underline{\mb{H}}_{n,p})$ is generated by $\chi_{\mr{diag}(p^{\delta_1},\dots,p^{\delta_n})}$ where the tuples run through the set $\{ (\delta_1,\dots,\delta_n) :0 \leq  \delta_1 \leq \dots \leq \delta_n \leq 1,, \ \delta_i \in \mb{Z} \}$. Following  \cite[Main theorem]{An70}, denote this space as $L^{+}$ and construct the isomorphism $ \omega' : L^{+} \rightarrow \mb{Q} [\omega'(\chi_1), \dots , \omega'(\chi_n)] $. Here, for $1 \leq i \leq n$, we have $\chi_i = \chi_{\mr{diag}(\underbrace{1,\dots , 1}_{n-i}, \underbrace{p, \dots , p }_i)}$ and $\omega'(\chi_i) = p^{-\frac{i(i-1)}{2}}s_i(x_1,\dots,x_n)$.

Define a map $f_p$ on polynomials by $f_p(h(x_1,x_2,\dots,x_n)) = h(px_1, \dots, px_n).$ The two spherical maps are connected by the following commutative diagram. \[   \begin{tikzcd}
    \underline{\mb{H}}_{n,p} \arrow{d}{\omega} \arrow{r}{\Psi} & L^{+} \arrow{d}{\omega'}   \\
     \mb{Q}[\omega(\pi_1), \dots , \omega(\pi_n)] \arrow{r}{f_p} &\mb{Q}[\omega'(\chi_1), \dots , \omega'(\chi_n)]
  \end{tikzcd}
\]    A direct computation shows that $f_p(\omega(\pi_i)) = \omega'(\chi_i), $ so all the maps in the above diagram are isomorphisms. Consequently, for $0 \leq \delta_1 \leq \cdots \leq \delta_n$, we obtain the relation $\omega(t(p^{\delta_1}, \dots , p^{\delta_n})) = \omega'(\chi(\delta_1, \dots , \delta_n))(p^{-1}x_1, \dots , p^{-1}x_n)$ where $\chi(\delta_1,\dots,\delta_n) = \chi_{\mr{diag}(p^{\delta_1},\dots,p^{\delta_n})}.$

The explicit formula for $ \omega'(\chi(\delta_1, \dots , \delta_n))(x_1, \dots , x_n)$ is constructed in \cite{An70}. For $0 \leq \delta_1 \leq \cdots \leq \delta_n$,
$$\omega'(\chi(\delta_1, \dots , \delta_n))(x_1, \dots , x_n) = \big(P^{k(\delta)}(p^{-1})\big)^{-1}p^{-\sum_i (n-i)\delta_i} Q_{(\delta)}(x_1,\dots , x_n),$$
where 
$$    Q_{\delta}(x_1,\dots,x_n) = \sum_{ \sigma \in S_n} x_{\sigma(1)}^{\delta_1}\dots x_{\sigma(n)}^{\delta_n} c(\sigma(x_1,\dots,x_n)), 
$$
\[
c(\sigma(x_1,\dots,x_n))  = \prod_{i=1}^{n-1} \prod_{j=i+1}^n \frac{1-p^{-1}x_{\sigma(i)}x_{\sigma(j)}^{-1}}{1-x_{\sigma(i)}x_{\sigma(j)}^{-1}},
\]
\begin{equation}\label{delta-k}
    \delta = (\underbrace{\delta_1, \dots , \delta_1}_{k_1}, \dots , \underbrace{\delta_t, \dots , \delta_t}_{k_t}), \quad k(\delta) = (k_1,\dots ,k_t)
\end{equation} and 
$$    P^{k(\delta)}(p^{-1}) = \frac{ \varphi_{k_1}(p^{-1}) \dots \varphi_{k_t}(p^{-1})}{(\varphi_{1}(p^{-1}))^n}.$$
Here, for $ m \in \mb{N}$, 
\begin{equation}\label{varphi}
    \varphi_m(x) := \prod_{i=1}^m (x^i-1) \ \text{and} \ \varphi_0(x) = 1.
\end{equation} 

Hence, 
\begin{align}\label{gln_sph}
    \omega(t(p^{\delta_1}, \dots , p^{\delta_n}))  & = \big(P^{k(\delta)}(p^{-1})\big)^{-1}p^{-\sum_i (n-i)\delta_i} Q_{(\delta)}(p^{-1}x_1,\dots , p^{-1}x_n) \nonumber \\
    & = \big(P^{k(\delta)}(p^{-1})\big)^{-1}p^{-\sum_i (n-i+1)\delta_i}Q_{(\delta)}(x_1,\dots , x_n).
\end{align}
\begin{prop} \label{O(T(p))-prop}
For $r \geq 1$, we have
\begin{equation}\label{O(T(p))}
    \Omega(T(p^r)) = x_0^r \sum_{0 \leq \delta_1 \leq \dots \delta_n \leq r } \big(P^{k(\delta)}(p^{-1})\big)^{-1} Q_{\delta}(x_1,\dots,x_n).
\end{equation}
    \end{prop}
\begin{proof} 
This follows from (\ref{T(p^r)2}) and (\ref{gln_sph}).
    \end{proof} 

\section{Eigenvalues of Ikeda lifts}
Let $f\in S_{2k}(\Gamma_1)$ be space of cusp forms over $\Gamma_1 = \mr{Sp}_{2}(\mb{Z})= \mr{SL}_2(\mb{Z})$. Let $F \in S_{k+n}(\Gamma_{2n})$ be the Ikeda lift of a Hecke eigenform  $f \in S_{2k}(\Gamma_1)$ as defined in \cite{Ik01}. If the Satake $p$-parameters of $f$ are $a_p,a_p^{-1}$, then the Satake $p$-parameters of $F$ are $\alpha_{0,p} = a_0 = p^{nk-\frac{n}{2}}a_p^{-n}$ and $\alpha_{i,p} = a_i = a_p p^{-n-\frac{1}{2}+i}$ for $1 \leq i \leq 2n$. Here, the parameters are normalized so that $a_0^2a_1\dots a_{2n}=p^{2nk-n}$. Since $\lambda_F(p^r) = \Omega(T(p^r))[a_0,a_1,\dots,a_{2n}]$, we substitute $x_i$ with $a_i$ in (\ref{O(T(p))}) and obtain 
\begin{equation}\label{eig-val}
    \lambda_F(p^r) = a_0^{r}\sum_{0 \leq \delta_1 \leq \dots \delta_{2n} \leq r }\big(P^{k(\delta)}(p^{-1})\big)^{-1}Q_{\delta}(a_1,\dots,a_{2n})
\end{equation}

where \[
 Q_{\delta}(a_1,\dots,a_{2n}) = \sum_{ \sigma \in S_{2n}} a_{\sigma(1)}^{\delta_1}\dots a_{\sigma(2n)}^{\delta_{2n}} c(\sigma(a_1,\dots,a_{2n}))
\] and 
\[
c(\sigma(a_1,\dots,a_{2n}))=  \prod_{i=1}^{2n-1} \prod_{j=i+1}^{2n} \frac{1-p^{-1}a_{\sigma(i)}a_{\sigma(j)}^{-1}}{1-a_{\sigma(i)}a_{\sigma(j)}^{-1}}.
\]
In the following lemma, we show that all but the contribution from $\sigma=1$ to $Q_\delta$ is vanishing.
\begin{lemma}\label{c(sigma)}
  We have $c(\sigma(a_1,\dots,a_{2n})) = 0 $ for all $\sigma \in S_{2n}$ with $\sigma \neq 1$.    
\end{lemma} 
\begin{proof}
    Let $\tau = \sigma^{-1}$. Since $\tau \neq 1$, there exists a positive integer $m$ such that  $1 \leq m \leq 2n-1$ and $\tau(m) > \tau(m+1)$. Set $j=\tau(m)$ and $i = \tau(m+1)$. Then $1 \leq i < j \leq 2n$ and 
    \[
    \sigma(i) = \sigma(\tau(m+1)) = m+1 = \sigma(j)+1.
    \] 
    Hence, 
    \[
1-p^{-1}a_{\sigma(i)}a_{\sigma(j)}^{-1} = 1-p^{-1}a_pp^{-n-\frac{1}{2}+\sigma(i)}a_p^{-1}p^{n+\frac{1}{2}-\sigma(j)}=1-p^{-1+\sigma(i)-\sigma(j)}=0.
\] 
This implies that $c(\sigma(a_1,\dots,a_{2n})) = 0$.
\end{proof}
The above lemma allows us to conclude that, for any $\delta$, we obtain 
\[
Q_\delta(a_1,\dots,a_{2n}) = a_1^{\delta_1} \dots a_{2n}^{\delta_{2n}} \prod_{i=1}^{2n-1} \prod_{j=i+1}^{2n}  \frac{1-p^{-1}a_{i}a_j^{-1}}{1-a_{i}a_{j}^{-1}}.
\]
 Hence, 
\[
\lambda_F(p^r) =  a_0^{r}\sum_{0 \leq \delta_1 \leq \dots \delta_{2n} \leq r }\big(P^{k(\delta)}(p^{-1})\big)^{-1}a_1^{\delta_1} \dots a_{2n}^{\delta_{2n}} \prod_{i=1}^{2n-1} \prod_{j=i+1}^{2n}  \frac{1-p^{-1}a_{i}a_j^{-1}}{1-a_{i}a_{j}^{-1}}.
\]
\begin{lemma}\label{P.c}
For any $(\delta)$, we have
   $$ \big(P^{k(\delta)}(p^{-1})\big)^{-1}\prod_{i=1}^{2n-1} \prod_{j=i+1}^{2n}  \frac{1-p^{-1}a_{i}a_j^{-1}}{1-a_{i}a_{j}^{-1}} = \frac{\varphi_{2n}(p^{-1})}{\varphi_{k_1}(p^{-1})\dots \varphi_{k_t}(p^{-1})}.$$
\end{lemma}
\begin{proof}
 For a fixed $i$ we have
    \[
    \prod_{j=i+1}^{2n}  \frac{1-p^{-1}a_{i}a_j^{-1}}{1-a_{i}a_{j}^{-1}} = \prod_{j=i+1}^{2n}  \frac{1-p^{-1+i-j}}{1-p^{i-j}} = \frac{1-p^{-1+i-2n}}{1-p^{-1}}.
    \] Taking the product over $i$ we obtain
    \[
    \prod_{i=1}^{2n-1} \prod_{j=i+1}^{2n}  \frac{1-p^{-1}a_{i}a_j^{-1}}{1-a_{i}a_{j}^{-1}} =  \prod_{i=1}^{2n-1} \frac{1-p^{-1+i-2n}}{1-p^{-1}} = \frac{\varphi_{2n}(p^{-1})}{\varphi_1(p^{-1})^{2n}}.
    \]
    Hence, 
    \[
     \big(P^{k(\delta)}(p^{-1})\big)^{-1}\frac{\varphi_{2n}(p^{-1})}{\varphi_1(p^{-1})^{2n}} =   \frac{\varphi_{2n}(p^{-1})}{\varphi_{k_1}(p^{-1})\dots \varphi_{k_t}(p^{-1})}.
    \]
\end{proof}

The above results allow us to obtain a formula for $\lambda_F(p^r)$ that is much simpler than the general formula in (\ref{O(T(p))}) in the case that $F$ is an Ikeda lift.
\begin{thm}\label{eigenvalue-ikeda}
 Let $F \in S_{k+n}(\Gamma_{2n})$ be the Ikeda lift of a Hecke eigenform $f \in S_{2k}(\Gamma_1)$. Let the Satake $p$-parameters of $f$ be $a_p,a_p^{-1}$.  For any prime $p$ and $r \in \N$, we have
    \begin{equation}\label{lambdaf(p)}
        \lambda_F(p^r) = p^{r(nk-\frac{n}{2})}\sum_{0 \leq \delta_1 \leq \dots \delta_{2n} \leq r } a_p^{-nr+\sum \delta_i} p^{\sum (-n - \frac{1}{2}+i)\delta_i}  \frac{\varphi_{2n}(p^{-1})}{\varphi_{k_1}(p^{-1})\dots \varphi_{k_t}(p^{-1})}.
    \end{equation}
\end{thm}
\begin{proof}
    Substitute $a_0=p^{nk-\frac{n}{2}}a_p^{-n}$ and $a_i = a_pp^{-n-\frac{1}{2}+i}$ in Eq (\ref{eig-val}). Then the theorem follows from Lemmas \ref{c(sigma)} and \ref{P.c}.  
\end{proof}
\begin{remark}
    For $r=1$, we compare the formula of $\lambda_F(p)$ obtained in Theorem \ref{eigenvalue-ikeda} with the formula obtained in Lemma $3$ of \cite{GN25}. Let $F \in S_{k+1}(\Gamma_2)$ be the Saito Kurokawa lift of $f \in S_{2k}(\Gamma_1).$  Replace $r$ with $1$ and $n$ with $1$ in Eq (\ref{lambdaf(p)}) to get 
    \[
    \lambda_F(p) = p^{k-1/2} \sum_{0 \leq \delta_1 \leq \delta_2 \leq 1} a_p^{-1+\delta_1+\delta_2} p^\frac{\delta_2 - \delta_1}{2} \frac{\varphi_2(p^{-1})}{\varphi_{k_1}(p^{-1}) \varphi_{k_2}(p^{-1})}.
    \] There are $3$ possible ordered pairs for $(\delta_1,\delta_2)$ and they are $(0,0),$ $(0,1)$ and $(1,1).$ Evaluating the contribution of each pair gives  
    \begin{align*}
        \lambda_F(p) & = p^{k-1/2} ( a_p^{-1} + p^{1/2}(1+p^{-1}) + a_p) \\ 
        & = p^{k-1/2}(a_p+a^{-1}_p) + p^k +p^{k-1} 
    \end{align*}
    If $\lambda_f(p)$ is the eigenvalue of $f$ with respect to $T(p)$ then it is known that $\lambda_f(p) = p^{k-1/2}(a_p+a_p^{-1}).$ Hence, 
    \[
    \lambda_F(p) =  \lambda_f(p) + p^k+p^{k-1}
    \]
    
   This agrees with the formula in Lemma~3 of \cite{GN25} after replacing
$n$ by $2$ and $k$ by $k+1$ in the notation used there.
\end{remark}

We will now rewrite the Hecke eigenvalue $\lambda_F(p^r)$ of $F$ as a polynomial in $p^{\pm 1/2}$ and obtain precise information about the coefficients of the polynomial. We will need some combinatorial information for that. Let $(\delta) = (\delta_1,\dots,\delta_{2n})$ with $0 \leq \delta_1 \leq \cdots \leq \delta_{2n} \leq r$. For $\sigma \in S_{2n}$, let $(\delta)_\sigma := (\delta_{\sigma(1)}, \cdots, \delta_{\sigma(2n)})$. An inversion of $(\delta)_\sigma$ is a $4$-tuple $(i,j,\delta_{\sigma(i)},\delta_{\sigma(j)})$ such that $i<j$ and $\delta_{\sigma(i)}>\delta_{\sigma(j)}$. Let ${\rm inv}((\delta)_\sigma)$ be the number of inversions of $(\delta)_\sigma$.  Let $S(\delta) := \{\sigma \in S_{2n} : (\delta) = (\delta)_\sigma\}$, and let $S_{(\delta)}$ be a set of representatives for $S_{2n}/S(\delta)$.

Further, for a positive integer $j$, a weak composition of $j$ into $2n$ parts is a set $\{b_1, b_2, \cdots, b_{2n}\}$ such that $b_i \geq 0$ for all $i$ and $b_1+b_2+\cdots+b_{2n} = j$. Note that $\{b_1, b_2, \cdots, b_{2n}\}$ is a weak composition of $j$ into $2n$ parts if and only if $\{b_1+1, b_2+1, \cdots, b_{2n}+1\}$ is a partition of $j+2n$ into $2n$ parts. Hence the number of weak compositions of $j$ into $2n$ parts equals $\binom{j+2n-1}{2n-1}$.
\begin{prop}\label{c-delta-prop}
For $(\delta) = (\delta_1,\dots,\delta_{2n})$, we denote $\displaystyle{\frac{\varphi_{2n}(p^{-1})}{\varphi_{k_1}(p^{-1})\dots \varphi_{k_t}(p^{-1})}}$ by $\Phi(\delta)$. For $j \geq 0$, let $c_{\delta,j} = \#\{\sigma \in S_{(\delta)} : {\rm inv}((\delta)_\sigma) = j\}$ and $ c(2n,j) = c_{\delta_{2n}, j}$, where $\delta_{2n} = (1,2,\cdots, 2n)$. We have 
$$ \Phi(\delta) =  \sum\limits_{j = 0}^{n(2n-1)} c_{\delta,j}p^{-j}$$
 where 
   \begin{equation}\label{eq1}
     c_{\delta,j} \leq c(2n,j).
    \end{equation}
\end{prop} 
\begin{proof}
It follows from Proposition 1.7.1 of \cite{Ri12} that 
$$ \Phi(\delta) =  \sum_{\sigma \in S_{(\delta)}} p^{-{\rm inv}((\delta)_\sigma)} \ \text{and} \ \frac{\varphi_{2n}(p^{-1})}{(1-p^{-1})^{2n}} = \sum_{\sigma \in S_{2n}} p^{\mr{-inv}(\sigma)}.$$
 Hence, we get \begin{equation}\label{Phi-poly}
    \Phi(\delta) =  \sum\limits_{j \geq 0} c_{\delta,j}p^{-j} \ \text{and} \ \frac{\varphi_{2n}(p^{-1})}{(1-p^{-1})^{2n}} = \sum_{j \geq 0} c(2n,j) p^{-j}.
\end{equation} From the definition, it is clear that $c_{\delta, 0} = 1$. Since the maximum number of inversions for any permutation in $S_{2n}$ is $\binom{2n}{2}=n(2n-1)$, we see that $c_{\delta,j} = 0$ for all $j > n(2n-1)$. Hence, the degree of $\Phi(\delta)$ is bounded above by $n(2n-1)$. 
Observe that, \begin{equation}\label{poly2}
    \displaystyle{\frac{\varphi_{k_1}(p^{-1})\dots \varphi_{k_t}(p^{-1})}{(1-p^{-1})^{2n}} = \sum_{ j \geq 0} b_jp^{-j}}
\end{equation} with $b_j \geq 0$ for all $j \geq 0.$ From (\ref{Phi-poly}) and (\ref{poly2}), we get that 
\[
\sum_{j \geq 0} c(2n,j) p^{-j} = \sum_{j \geq 0} c_{\delta,j}p^{-j} . \sum_{j \geq 0} b_jp^{-j}
\] Compare the coefficients to see that, for $j \geq 0$ 
\begin{equation}\label{c-delta-ineq}
    c(2n,j) = \sum_{l=0}^j c_{\delta,j-l}b_l \geq c_{\delta,j}b_0 = c_{\delta,j}.
\end{equation}

\end{proof}

\begin{prop}\label{lamda_F(p^r)} 
If $\tilde{\lambda}_F(p^r) = \frac{\lambda_F(p^r)}{p^{r(nk-\frac{n}{2})}}$ then, 
    \[
    \tilde{\lambda}_F(p^r) = \sum_{m \in Z} c_mp^{\frac{m}{2}}
    \] where the coefficients $c_m$ satisfy the following.
    \begin{enumerate}
    \item $c_m = 0$ for all $ m > rn^2$ or $m \leq -2n(2n-1)$,
    \item  $c_{rn^2} =1$ and 
    \item $\lvert c_m \rvert \le \sum_{j=0}^{\frac{rn^2-m}{2}} \binom{rn^2-m-2j+2n-1}{2n-1}c(2n,j)$ for all $-2n(2n-1) < m < rn^2$.
    \end{enumerate}
\end{prop} 
\begin{proof}
    Set $E(\delta) := \sum_{i=1}^{2n} (-n-\frac{1}{2}+i)\delta_i$. From Theorem \ref{eigenvalue-ikeda} and Proposition \ref{c-delta-prop}, we have $\tilde{\lambda}_F(p^r) = \sum_{m \in Z} c_mp^{\frac{m}{2}}$ with the coefficients $\ds{c_m = \sum_{j \geq 0} \sum_{\delta:E(\delta)-j = \frac{m}{2}}a_p^{-nr+\sum \delta_i} c_{\delta,j} }$. Since $\Phi(\delta)$ is a polynomial in $p^{-1}$ with non-zero constant term, to find the maximum power of $p$, it suffices to maximize $E(\delta)$.  Observe that 
    $$
        E(\delta)  = \sum_{i=1}^n  (-n-\frac{1}{2}+i)\delta_i + \sum_{i=n+1}^{2n} (-n-\frac{1}{2}+i)\delta_i 
    $$
    Substitute $i = 2n+1-i'$ in the second summation to get 
    \begin{align}\label{E-delta-value}
         E(\delta) & = \sum_{i=1}^n  (-n-\frac{1}{2}+i)\delta_i + \sum_{i'= n}^1(-n-\frac{1}{2}+2n +1 -i') \delta_{2n+1-i'} \nonumber \\ 
         & = \sum_{i=1}^n  (-n-\frac{1}{2}+i)\delta_i + \sum_{i'= 1}^n (n+\frac{1}{2}-i') \delta_{2n+1-i'} \nonumber \\
         & = \sum_{i=1}^n (n+\frac{1}{2}-i)(\delta_{2n+1-i} - \delta_i)
    \end{align}
   Since $ 0 \leq \delta_1 \leq \delta_2 \dots \leq \delta_{2n} \leq r $ and $n+\frac{1}{2}-i > 0$ for all $1 \leq i \leq n,$ $E(\delta) \geq 0$ for all $\delta.$ The maximum value of $\delta_{2n+1-i} - \delta_i$ for any $i$ is $r$ and it occurs for all $i$ for the tuple $\hat{\delta} =( \underbrace{0,\dots,0}_n,\underbrace{r,\dots,r}_n).$ Hence, the maximum value of $E(\delta)$ is 
   \[
\sum_{i=1}^{n}(n+\frac{1}{2}-i) r = \frac{rn^2}{2}.
\]
 The constant term of $\Phi(\hat\delta)$ is $1$ and the power of $a$ for $\hat\delta$ is $0$. Hence $c_{rn^2}=1$ and $c_m = 0$ for $m > rn^2$. On the other hand, the minimum value of $E(\delta)$ is $0$ and it occurs for the tuple $\hat{\hat{\delta}} = (\underbrace{r,r,\dots,r}_{2n}).$ Since $\Phi(\delta)$ is a polynomial in $p^{-1}$ with maximum possible degree $n(2n-1)$, we see that $c_m = 0$ for $m < -2n(2n-1)$, as required. For $m = -2n(2n-1)$, the only contribution to $c_m$ comes from $\delta$ such that $E(\delta) = 0$ and $j = n(2n-1)$. From (\ref{E-delta-value}), we see that we must have $\delta_{2n+1-i} = \delta_i$ for all $1 \leq i \leq n$. For such a $\delta$, it is easy to see that $c_{\delta, n(2n-1)} = 0$. Hence, we get $c_m = 0$ also for $m = -2n(2n-1)$.

Now let $m$ be such that $-2n(2n-1) < m < rn^2$. It follows from \cite{De74} that $\lvert a_p \rvert =1.$ Since $c_{\delta,j}$ are non negative integers, using (\ref{eq1}), we see that 
\begin{equation}\label{eqn11}
 \lvert c_m \rvert \leq \sum_{\substack{j \geq 0 \\ \delta : E(\delta) - j = \frac{m}{2}}} c_{\delta,j} = \sum\limits_{j=0}^{n(2n-1)} \sum\limits_{\delta : E(\delta)-j=m/2} c_{\delta,j} \leq \sum_{j=0}^{n(2n-1)} \# \{\delta : E(\delta)- j = \frac{m}{2}\} c(2n,j).
\end{equation}
Note that, for a fixed $m$, we have $\# \{\delta : E(\delta)- j = \frac{m}{2}\} = 0$ if $j > (rn^2-m)/2$. 
 Observe that \[
    rn^2 - 2E(\delta) = \sum_{i=1}^n (2n+1-2i)\delta_i + \sum_{i=n+1}^{2n} (2i-2n-1)(r-\delta_i). 
    \] Let $(2n+1-2i)\delta_i = x_i$ for $1 \leq i \leq n$ and $(2i-2n-1)(r - \delta_i) = x_{i}$ for $n+1 \leq i \leq 2n.$
    Hence, for every $\delta$ such that $E(\delta) - j = \frac{m}{2}$ there exists a tuple $(x_1,\dots ,x_{2n})$ such that $\sum_{i=1}^{2n} x_i =  rn^2-2E(\delta)=rn^2-m-2j$, i.e. we get a weak composition of $rn^2-m-2j$ into $2n$ parts. The number of such tuples is bounded by $\ds{\binom{rn^2-m-2j+2n-1}{2n-1}}.$ Since $\delta \to (x_1, \cdots, x_{2n})$ is an injective map, 
    \begin{equation}\label{eq2} 
        \# \{ \delta : E(\delta)-j  = \frac{m}{2} \} \leq  \ds{\binom{rn^2-m-2j+2n-1}{2n-1}} .
    \end{equation}
  Substituting (\ref{eq2}) in (\ref{eqn11}), we obtain 
   \begin{equation}\label{cm-ine}
        \lvert c_m \rvert \leq \sum_{j=0}^{\frac{rn^2-m}{2}} \binom{rn^2-m-2j+2n-1}{2n-1}c(2n,j)
   \end{equation} for all $-2n(2n-1) < m < rn^2$. 
\end{proof}
\begin{thm}\label{positivity}
    Let $F \in S_{k+n}(\Gamma_{2n})$ be the Ikeda lift of a Hecke eigenform $f \in S_{2k}(\Gamma_1).$ Let $\lambda_F(p^r)$ be the eigenvalue of $F$ for the Hecke operator $T(p^r)$. For $r \ge 1$, we have $\lambda_F(p^r)>0$ for all $p >14n^2$.
\end{thm}
\begin{proof} 
Since,

\begin{align}\label{lambda_1}
\widetilde{\lambda}_F(p^r) =p^{rn^2/2}+\sum_{m=-2n(2n-1)}^{rn^2-1} c_m p^{m/2} = p^{rn^2/2} \bigg( 1 + \sum_{m=-2n(2n-1)}^{rn^2-1} c_m p^{(m-rn^2)/2} \bigg). 
\end{align} From (\ref{cm-ine}), it follows that 
\begin{align*}
\left| \sum_{m=-2n(2n-1)}^{rn^2-1}c_m p^{\frac{m-rn^2}{2}}\right|
&\leq\sum_{m=-2n(2n-1)}^{rn^2-1}\sum_{j=0}^{\frac{rn^2-m}{2}}c(2n,j)\binom{rn^2-m-2j+2n-1}{2n-1}p^{\frac{m-rn^2}{2}}.
\end{align*} Let 
\[
S = \sum_{m=-2n(2n-1)}^{rn^2-1}\sum_{j=0}^{\frac{rn^2-m}{2}}c(2n,j)\binom{rn^2-m-2j+2n-1}{2n-1}p^{\frac{m-rn^2}{2}}.
\] From (\ref{lambda_1}), it is clear that to show $\widetilde{\lambda}_F(p^r)>0$ it is enough to show that $S < 1.$
We now make the change of variables $a=rn^2-m$ and observe that 
\[
S=\sum_{a=1}^{rn^2+2n(2n-1)}\sum_{0\leq j\leq a/2}c(2n,j)\binom{a-2j+2n-1}{2n-1}p^{-a/2}.
\]
Then, 
\begin{align*}
    1+S & =\sum_{a=0}^{rn^2+2n(2n-1)}\sum_{j=0}^{a/2}c(2n,j)\binom{a-2j+2n-1}{2n-1}p^{-a/2} \\
        &= \sum_{j=0}^{\frac{rn^2+2n(2n-1)}{2}} \sum_{a=2j}^{rn^2+2n(2n-1)} c(2n,j)\binom{a-2j+2n-1}{2n-1}p^{-a/2}.
\end{align*}
Let $\ell=a-2j$ and we can rewrite the above sum as 
\[
1+S = \sum_{j=0}^{\frac{rn^2+2n(2n-1)}{2}} \sum_{\ell=0}^{rn^2+2n(2n-1)-2j} c(2n,j)\binom{\ell+2n-1}{2n-1}p^{-\ell/2}p^{-j}.
\] Since all the terms are nonnegative, extending the ranges of summation gives 
\begin{align*}
1+S&\leq\left(\sum_{\ell\geq 0}\binom{\ell+2n-1}{2n-1}p^{-\ell/2}
\right)\left(\sum_{j\geq 0}c(2n,j)p^{-j}\right). \\ 
& = \frac{1}{(1-p^{-1/2})^{2n}} \prod_{i=1}^{2n} (1+p^{-1}+p^{-2}+\dots+p^{-(i-1)}) < 2.
\end{align*}The last inequality follows from Lemma \ref{f(x)ineq} for $x = p^{-1/2}.$ Hence, for all $ p > 14n^2$ and $n \ge 2$ we get that 
$S < 1$ and it implies $\widetilde{\lambda}_F(p^r)>0.$ The case of $n=1$ is proved in \cite{Br99}.
\end{proof} 
\begin{lemma}\label{f(x)ineq}
    If \[
    f(x) = \frac{1}{(1-x)^{2n}}\prod_{i=1}^{2n}\left(1+x^2+\cdots+x^{2(i-1)}\right).
    \] then $f(x) < 2 $ for all $nx \le 1/\sqrt{14} $ and $n \ge 2.$
\end{lemma}
\begin{proof}
For \(0<x<1\), we have \( 1+x^2+\cdots+x^{2(i-1)}<\frac{1}{1-x^2}.\)
This implies that
\[
f(x)<\frac{1}{(1-x)^{2n}(1-x^2)^{2n-1}}.
\]
Bernoulli's inequality implies that \( (1-x^2)^{2n-1}\geq1-(2n-1)x^2. \) Since \(x\leq 1/\sqrt{14}n,\) we have
\[
(2n-1)x^2\leq\frac{2n-1}{14n^2}.
\]
The function \( \frac{2n-1}{14n^2} \) is decreasing for \(n\geq 1\). In particular, for \(n\geq 2\),
\[ 
\frac{2n-1}{14n^2} \leq \frac{3}{56}.
\] Hence,
\begin{equation}\label{ineq1}
    1-(2n-1)x^2 \geq \frac{53}{56}.
\end{equation}
Also, by Bernoulli's inequality,
\[
(1-x)^n \geq 1-nx.
\]
Therefore,
\[
(1-x)^{2n} = \bigl((1-x)^n\bigr)^2 \geq (1-nx)^2.
\]
Since \( nx\leq \frac{1}{\sqrt{14}},\) we get that 
\begin{equation}\label{ineq2}
  (1-x)^{2n}\geq\left(1-\frac{1}{\sqrt{14}}\right)^2.  
\end{equation}
Combining (\ref{ineq1}) and (\ref{ineq2}), we obtain
\[
(1-x)^{2n}(1-x^2)^{2n-1}\geq\left(1-\frac{1}{\sqrt{14}}\right)^2 \frac{53}{56} > \frac{1}{2}.
\]
Hence, 
\[
f(x)<\frac{1}{(1-x)^{2n}(1-x^2)^{2n-1}} < 2 
\] for all $nx \le 1/\sqrt{14}$ and $n \geq 2.$
\end{proof}

\bibliography{references}
\bibliographystyle{acm}

\end{document}